\documentclass[11pt]{article}
\usepackage{amssymb,amsmath}
\usepackage{amsthm}
\newtheorem{theorem}{Theorem}[section]
\newtheorem{theorem*}{Theorem A\!\!}
\newtheorem{proposition}{Proposition}[section]
\newtheorem{proposition*}{Proposition A\!\!}

\newtheorem{corollary*}{Corollary A\!\!}

\newtheorem{lemma}{Lemma}[section]

\DeclareMathOperator{\ad}{ad}

\begin{document}

\title{Singular conformally invariant trilinear forms and covariant differential operators on the sphere}

\author{ Jean-Louis Clerc}
\date{December 10, 2010}
\maketitle
\begin{abstract}

Let $G=SO_0(1,n)$ be the conformal group acting on the $(n-1)$ dimensional sphere $S$, and  let $(\pi_\lambda)_{ \lambda\in \mathbb C}$ be the spherical principal series. For generic values of $\boldsymbol \lambda =(\lambda_1,\lambda_2,\lambda_3)$ in $\mathbb C^3$, there exists a (essentially  unique) trilinear form on $\mathcal C^\infty(S)\times \mathcal C^\infty(S)\times \mathcal C^\infty(S)$ which is invariant under $\pi_{\lambda_1}\otimes \pi_{\lambda_2}\otimes \pi_{\lambda_3}$. Using differential operators on the sphere $S$ which are covariant under the conformal group $SO_0(1,n)$, we construct new invariant trilinear forms  corresponding to singular values of $\boldsymbol \lambda$. The new forms are shown to be certain residues of the family of generic forms (which depends meromorphically on $\boldsymbol \lambda$).
\medskip

Soit $G=SO_0(1,n)$ le groupe conforme de la sph\`ere $S$ de dimension $(n-1)$. Soit $(\pi_\lambda)_{\lambda\in \mathbb C}$ la s\'erie principale sph\'erique associ\'ee. Pour des valeurs g\'en\'eriques de $\boldsymbol \lambda =(\lambda_1,\lambda_2,\lambda_3)$ dans $\mathbb C^3$, il existe une forme trilin\'eaire (essentiellement unique) sur $\mathcal C^\infty(S)\times \mathcal C^\infty(S)\times \mathcal C^\infty(S)$ qui est invariante par $\pi_{\lambda_1}\otimes \pi_{\lambda_2}\otimes \pi_{\lambda_3}$. En utilisant les op\'erateurs 
diff\'erentiels sur $S$ qui sont conform\'ement covariants, nous construisons de nouvelles formes trilin\'eaires invariantes correspondant ˆ des valeurs singuli\`eres de $\boldsymbol \lambda$. On montre que ces nouvelles formes sont des r\'esidus de la famille des formes g\'en\'eriques (qui d\'epend m\'eromorphiquement de $\boldsymbol \lambda$).

\end{abstract}

\footnotemark[0]{2000 Mathematics Subject Classification : 22E45, 43A85}

\section*{Introduction}

Let $G=SO_0(1,n)$, acting conformally on the $(n-1)$-dimensional sphere $S$. In a previous paper, coauthored with B. \O rsted (see \cite{co}), we introduced a trilinear from on $\mathcal C^\infty(S)\times\mathcal C^\infty(S) \times \mathcal C^\infty(S)$ which is invariant under $\pi_{\lambda_1}\otimes \pi_{\lambda_2} \otimes \pi_{\lambda_3}$, where $\lambda_1,\lambda_2,\lambda_3$ are three complex parameters, and for $\lambda\in \mathbb C$, $\pi_\lambda$ is the spherical (nonunitary) principal series representation of $S0_0(1,n)$, realized on $\mathcal C^\infty(S)$. The form is constructed via meromorphic continuation in the parameter $\boldsymbol \lambda = (\lambda_1,\lambda_2,\lambda_3)$ in $\mathbb C^3$, with  explicitely described  simple poles. In the present paper, we determine (some of) the residues of this meromorphic family, producing new conformally invariant trilinear forms. Their expressions involve \emph{covariant differential operators} on $\mathcal C^\infty(S)$.

The geometric and functional context is presented in section 1. The two next sections (sections 2, 3) are devoted to a brief study of covariant differential operators on the sphere, for which we could not find an adequate reference. In section 2, we give (following Kostant and Lepowsky) necessary conditions for the existence of such a covariant operator. In section 3, we study the meromorphic continuation (in the parameter $s$) of the distribution $\vert x-y\vert^s$ on $S\times S$. It has simple poles at $s=-(n-1)-2k, k\in \mathbb Z$ and its residues at the poles are determined, leading to the construction of (all) covariant differential operators on $S$. In section 4, we use these covariant differential operators to construct new families of trilinear invariant forms on $\mathcal C^\infty(S)\times \mathcal C^\infty(S)\times \mathcal C^\infty(S)$. Viewed as distributions on $S\times S\times S$, they are \emph{singular} (supported in lower dimensional submanifolds). The families are indexed by an integer $k$ in $\mathbb N$, and depend meromorphically on two complex parameters. In section 5  we first recall the construction (cf \cite{co}) of the meromorphic family of trilinear forms $(\mathcal K_{\boldsymbol \alpha})_{\boldsymbol\alpha\in \mathbb C^3}$. The set of poles is a union of  planes in $\mathbb C^3$, coming in four families of parallel and equidistant planes. The new forms are shown to be (essentially) the residues along the planes belonging to three of the four families. The residues corresponding to the last family are even more singular (the corresponding distributions are supported by the "diagonal" in $S\times S\times S$) and we delete their study to a (hopefully) near future.

\section{The geometric context}

Let $S=S^{n-1}$ be the unit sphere in $\mathbb R^n$,
\[ S = \{ x=(x_1,x_2,\dots,x_n)\ ; \vert x\vert = \langle x,x\rangle^{\frac{1}{2}} = (x_1^2+x_2^2+\dots +x_n^2)^{\frac{1}{2}} = 1\}\]
equipped with the Riemmannian metric naturally induced from the Euclidean structure of $\mathbb R^n$. We assume $n\geq 3$, although most of the statements are (or could be made) valid for $n=2$ (see \cite{co} for similar remarks).

The group $K=SO(n,\mathbb R)$ operates transitively on $S$ by isometries. However, our interest is in conformal geometry of the sphere, for which another model of the sphere is more fitted. Let $\mathbb R^{1,n}$ be the real vector space of dimension $n+1$ equipped with the Lorentzian quadratic form
\[[y,y] =y_0^2-(y_1^2+\dots+y_n^2)
\]
and let $\mathcal S$ be the set of all isotropic lines in $\mathbb R^{1,n}$, viewed as a closed submanifold of the real $n$-dimensional projective space. Then the mapping
\[x\mapsto \mathbb R \,(1,x)\qquad S\longmapsto \mathcal S
\] is a 1-1 correspondance, which is easily seen to be a diffeomorphism. The group $G=SO_o(1,n)$ operates naturally on $\mathcal S$, and this action can be transferred to an action of $G$ on $S$. If $x$ is an element of $S$, and $g$ belongs to $G$, then $g(x)$ is determined by the following equality in $\mathbb R^{1,n}$ :
\[ (1,g(x)) = \frac{1}{(g. (1,x))_0}\,g. (1,x) \ .
\]

The group $K$ can be viewed as a closed subgroup of $G$ and is a maximal compact subgroup of $G$. Introduce the base point ${\bf 1} = (1,0,\dots,0)$ in $S$. Then the stabilizer of $\bf 1$ in $G$ is the parabolic subgroup $P=MAN$, where $M$ is the stabilizer of $\bf 1$ in $K$ (isomorphic to $SO(n-1)$), $A$ is the Abelian 1-dimensional subgroup given by
\[A = \left\{ a_t = \begin{pmatrix}\cosh t&\sinh t&0&\dots&0\\\sinh t& \cosh t&0&\dots&0\\0&0&1 & &\\ \vdots&\vdots&&\ddots&\\0&0&&&1\end{pmatrix},\quad t\in \mathbb R\right\}
\]
and
\[N= \left\{\ n_\xi =\begin{pmatrix}1+\frac{\vert\xi\vert^2}{2}&-\frac{\vert\xi\vert^2}{2}&&\xi^t&\\\frac{\vert \xi\vert^ 2}{2}&1-\frac{\vert \xi\vert^2}{2}&&\xi^t&\\
&&1&&
\\\xi&-\xi&&\ddots&\\&&&&1
\end{pmatrix},\quad \xi\in \mathbb R^{n-1}\right\}\ .
\]
The element $a_t$ ($t\in \mathbb R$) acts on $S$ by
\[ a_t \,\begin{pmatrix}x_1\\x_2\\\dots\\x_n\end{pmatrix} = \begin{pmatrix} \frac{\sinh t +x_1 \cosh t}{\cosh t + x_1\sinh t }\\ \frac{x_2}{\cosh t + x_1\sinh t} \\ \vdots\\ \frac{x_n}{\cosh t + x_1\sinh t}\end{pmatrix}\ .
 \]
 Let $\mathfrak g = Lie(G)$ be the Lie algebra of $G$, and let $\mathfrak {k,m,a,n}$ be  the Lie subalgebras of (respectively) $K,M,A,N$.
 Let
 \[ H = \begin{pmatrix}0&1&0&\dots&0\\1&0&0&\dots&0\\0&0&0&\dots&0\\\vdots&\vdots&\vdots&\vdots&\vdots\\0&0&0&\dots&0\end{pmatrix}\in \mathfrak {gl}(n+1,\mathbb R)
\]
be a generator of $\mathfrak a$.
 
 Let $\overline {\mathfrak n}= \theta(\mathfrak n)$, where $\theta$ is the standard Cartan involution of $\mathfrak g$ given by $\theta X = -X^t$. The Lie algebra $\mathfrak g$ admits the following decomposition   \[\mathfrak g = \mathfrak {\overline n}\oplus \mathfrak m\oplus \mathfrak a \oplus \mathfrak n\ ,\]
which is the eigenspace decomposition for the action of  $\ad H$, with eigenvalue $0$ on $\mathfrak{m\oplus a}$, $1$ on $\mathfrak n$ and $-1$ on $\overline {\mathfrak n}$.
Let also $\mathfrak p = \mathfrak m\oplus \mathfrak a \oplus \mathfrak n$ be the parabolic subalgebra corresponding to $P$.

The action of $G$ turns out to be \emph{conformal}, that is, for any $g$ in $G$ and $x$ in $S$, the differential $ Dg(x)$ satisfies, for any $\xi$ in the tangent space $T_xS$
\begin{equation}\label{covinf}
\vert Dg(x)\xi\vert =\kappa(g,x) \vert \xi\vert\ ,\end{equation}
where $\kappa(g,x)$ is a positive constant, called the \emph{conformal factor} of $g$ at $x$.

\begin{proposition}\label{conffac} The conformal factor $\kappa(g,x)$ is a smooth function of both $g$ and $x$, which satisfies moreover the following properties : 
\medskip

$i)$(cocycle property)
\begin{equation}\label{coc}
\forall g_1,g_2\in G, x\in S,\quad  \kappa(g_1g_2,x) = \kappa(g_1,g_2(x))\,\kappa(g_2,x)\end{equation}

$ii)$ $\forall g\in G, x\in S\quad  \kappa(g,g^{-1}(x)) = \kappa(g^{-1},x)^{-1}$

$iii)$  $\forall x\in S, k\in K\quad \kappa(k,x)=1$

$iv)$ $\forall x\in S, t\in \mathbb R,\quad  \kappa(a_t,x) = (\cosh t +x_1\sinh t)^{-1}$.

\end{proposition}
Let $g$ in $G$. As the dimension of the tangent space $T_xS$ is $n-1$, the Jacobian of  $g$  at $x$ (with respect to the Euclidian measure on $S$) is given by
\begin{equation*}
j(g,x) = \kappa(g,x)^{n-1}\ .
\end{equation*}
The corresponding change of variable formula is
\begin{equation}\label{varchange}
\int_S f(g^{-1}(x)) \,d\,\sigma(x) = \int_S f(y)\kappa(g,y)^{n-1} d\,\sigma(y)\ ,
\end{equation}
where $d\sigma$ is the Lebesgue measure on $S$.

The Euclidean distance, restricted to $S\times S$, satisfies an important covariance property under the action of $G$.

\begin{proposition} Let $g$ in $G$ and $x,y$ in $S$. Then
\begin{equation}\label{cov}
\vert g(x)-g(y)\vert = \kappa(g,x)^{\frac{1}{2}}\ \vert x-y\vert \ \kappa(g,y)^{\frac{1}{2}}\ .
\end{equation}
\end{proposition}
Notice that \eqref{covinf} can be viewed as the infinitesimal form of \eqref{cov}.

\section{Conformal representations, associated Verma modules and covariant differential operators}

To the conformal action of $G$ on $S$ is associated a family of representations. For convenience set $\rho = \frac{n-1}{2}$. Now let $\lambda$ be any complex number and define for any $g$ in $G$ and $f$ in $\mathcal C^\infty(S)$
\begin{equation}
\pi_\lambda(g)f\,(x) = \kappa (g^{-1},x)^{\rho+\lambda} f\big(g^{-1}(x)\big)
\end{equation}
for $x$ in $S$. Thanks to the cocycle property of $\kappa$ (see \eqref{coc}), $\pi_\lambda$ is a representation of $G$ on the space $\mathcal C^\infty(S)$, which is continuous for the usual topology on $\mathcal C^\infty(S)$. For a detailed study of these representations, see \cite{tak}. They are also called the \emph{spherical (non unitary) principal series} of $G$.

The reason for putting  $\rho$ in the parameter of the representation is to introduce more  symmetry in the formulation of the following duality result.
\begin{proposition}
For $f$ and $\varphi$ in $\mathcal C^\infty(S)$, and for any $g\in G$,
\begin{equation}\label{dual}
\int_S \pi_\lambda(g)f(x) \varphi(x)d\sigma(x) = \int_S f(x)\pi_{-\lambda}(g^{-1})\varphi(x) d\sigma(x)\ .
\end{equation}
\end{proposition}
This is a consequence of the change of variable formula \eqref{varchange}.

This can be used to extend the definition of the representation $\pi_\lambda$ to the space of distributions $\mathcal C^{-\infty}(S)$ by setting for $F$ in $\mathcal C^{-\infty}(S)$
\begin{equation}
\big(\pi_\lambda(g) F,\varphi\big): = (F, \pi_{-\lambda}(g^{-1})\varphi)
\end{equation}
for any $\varphi$ in $\mathcal C^\infty(S)$. Let $\delta=\delta_{\bf 1}$ be the Dirac measure at $\bf 1$. Then
\begin{equation}\label{dirac}
\pi_\lambda(g) \delta = \kappa(g,{\bf 1})^{\rho-\lambda} \delta_{g({\bf 1})}\ .
\end{equation}

As usual, it is possible to differentiate the action of $G$ on $\mathcal C^\infty(S)$ and obtain a representation (still denoted by $\pi_\lambda$) of the Lie algebra $\mathfrak g$ or of the universal envelopping algebra $\mathfrak {U(g)}$. The envelopping algebra acts by differential operators on $\mathcal C^\infty(S)$, and also on $\mathcal C^{-\infty}(S)$. Hence this action preserves the support of the distributions.

Let $\mathcal V$ be the space of distributions on $S$ supported at $\{\bf 1\}$. The envelopping algebra $\mathfrak {U(g)}$ acts on $\mathcal V$ through $\pi_\lambda$, and we let $\mathcal V_\lambda$ be the corresponding $\mathfrak {U(g)}$-module. By differentiation of \eqref{dirac}, 
\begin{equation}\label{verma1}
 \pi_\lambda(X) \delta = 0, \qquad{ \rm for}\ X\in \mathfrak {m\oplus n}
 \end{equation}
\begin{equation}\label{verma2}
 \pi_\lambda(H) \delta = (\lambda-\rho)\,\delta\qquad 
\end{equation}
Moreover, by the Poincar\'e-Birkhoff-Witt theorem, the space $\mathcal V_\lambda$ is identified with $\mathfrak {U(\overline n)}\delta$. This remark, together with  \eqref{verma1} and \eqref{verma2} show that $\mathcal V_\lambda$ is a \emph{generalized Verma module}, in the sense of Lepowsky (see \cite{lep1}, \cite{lep2}, or \cite{d}).

Let $\lambda,\mu$ be two  complex numbers. A differential operator $D$ (with smooth coefficients) on $S$ is said to be \emph{conformally covariant} with respect to the representations $(\pi_\lambda, \pi_\mu)$ if, for any $g$ in $G$
\begin{equation}\label{covdiff}
D\circ\pi_\lambda(g)=\pi_\mu(g)\circ D\ .
\end{equation}

\begin{proposition} Let $D$ be a differential operator on $S$, which is not identically $0$, and assume $D$ is conformally covariant w.r.t.  $(\pi_\lambda,\pi_\mu)$.
Then, either $\lambda=\mu$ (and then $D$ is a multiple of the identity), or $\mu = -\lambda= k$ for some $k$ in $\mathbb N$.
\end{proposition}
\begin{proof} Let $D$ be such differential operator. As $D$ commutes to the action of $K$ (which does not depend on $\lambda$), $D$ is a polynomial in the Laplacian $\Delta$ (a special case of the characterization of the invariant differential operators on a Riemannian symmetric space). As $D\neq 0$, $D\delta$ is already different from $0$.
Now $D$ acts on distributions on $S$, preserving the support. Hence $D$ yields by restriction an operator on $\mathcal V$ and the covariance property \eqref{covdiff}  shows that $D$ induces a homomorphism of $\mathfrak {U(g)}$-modules from $\mathcal V_\lambda$ into $\mathcal V_\mu$. Moreover, $D\delta$ is a non trivial conical vector (i.e. killed by $\mathfrak {m\oplus n}$) of weight $\lambda-\rho$ in $\mathcal V_\mu$. Now the homomorphisms of generalized Verma modules and the conical vectors were studied by Lepowsky (see \cite{lep1}, \cite{lep2}) and the only possibilities for the existence of a non trivial homomorphism between $\mathcal V_\lambda$ and $\mathcal V_\mu$ are $\lambda=\mu$ or $\mu=-\lambda=k$, for some $k$ in $\mathbb N$ (\cite{lep2}).
\end{proof}
\noindent
{\bf Remark.} The quoted reference shows that the necessary conditions on $\lambda$ and $\mu$ for the existence of a homomorphism from $\mathcal V_\lambda$ into $\mathcal V_\mu$ are also sufficient. Moreover the homomorphism is then unique, up to a constant. For any $k$ in $\mathbb N$, we will now construct  an explicit covariant differential operator  on $S$ which yields a non trivial homomorphism of $\mathcal V_{-k}$ into $\mathcal V_k$.

\section{ The meromorphic continuation of $\vert {\bf 1}-x\vert^{s}$ and its residues}

For $s$ in $\mathbb C$, let  $h_s(x) = \vert {\bf 1} -x\vert^s$. This defines a smooth function on $ S$ outside of the point $\bf 1$. For $\Re s>-(n-1)$, the function is integrable and  will be considered as a distribution on $S$, still denoted by $h_s$. It depends holomorphically on $s$. We want to show that it can be extended meromorphically to  $\mathbb C$. The main ingredient to do this is the \emph{Bernstein-Sato identity} ({\it stricto sensu}, one should consider  the Bernstein-Sato identity for the smooth function $\vert 1-x\vert^2$ on $S$).

\begin{proposition} The following identity holds on $S\setminus \{\bf 1\}$
\begin{equation}\label{bern}
[\Delta +\frac{s}{2}(\frac{s}{2}+n-2)]\,h_s = s(s+n-3)\,h_{s-2}
\end{equation}

\end{proposition}
\begin{proof}
A function $f$ on $S$ which is invariant under the subgroup $M$ depends only on the distance from $x$ to ${\bf 1}$ and  can be written in a unique way as  $f(x) =\varphi(\theta)$, where $\theta = \arccos \,\langle x,{\bf  1}\rangle$, and $\varphi$ is a function defined on the interval $[0,\pi]$. As the Laplacian commutes to the rotations, the function $\Delta f$ is also invariant under $M$, and the following relation holds :
\begin{equation*}
\Delta f (x) = \varphi''(\theta) + \frac{(n-2)}{ \tan \theta} \, \varphi'(\theta)\ .
\end{equation*}

 As $\vert {\bf 1} - x\vert^2 = 2(1-\cos \theta)$, \eqref{bern} follows easily. 
\end{proof}

\begin{proposition}\label{mero0} 
The function $s\mapsto h_s$ originally defined for $\Re s >-(n-1)$ can be extended as a (distribution-valued) meromorphic function on $\mathbb C$, with simple poles at $s = -(n-1)-2k, k\in   \mathbb N$.
\end{proposition}
The proof is standard and uses mainly integration by parts in the  form

\[\Big( h_s \, ,\big(\Delta + \frac{s}{2}(\frac{s}{2}+n-2)\big) f\Big) =\Big (\big(\Delta + \frac{s}{2}(\frac{s}{2}+n-2)\big) h_s \,,f\Big)\ ,
\]
where $f$ is any smooth test function.
If the left handside is already defined, and $s(s+n-3)\neq 0$, it can be used to define $(h_{s-2}\,, f)$. The meromorphic dependance on $s$ is easy to verify.

A variant of  Propostion \ref{mero0} will be used later on. 
\begin{proposition}\label{mero1} Let $\varphi$ be in $\mathcal C^k(S)$ (i.e. $k$-times continuously differentiable), then the function  $s\mapsto \int \varphi(x) h_s(x) \, d\sigma(x)$ can be extended meromorphically in the open set $\Re(s)<-(n-1)-2\,[\frac{k}{2}]$ with at most simple poles at $s = -(n-1)-2l, l\in \mathbb N, l< [\frac{k}{2}]$.
\end{proposition}

Denote by $r_k$ the residue at $-(n-1)-2k$ of the distribution-valued function  $s\mapsto h_s$.

\begin{proposition}
\begin{equation}\label{res0}
r_0 := Res(h_s, -(n-1)) = \frac{\pi^\rho}{\Gamma(\rho)} \ \delta_{\mathbf 1}\ .
\end{equation}
\end{proposition}

\begin{proof}
\[(h_s,f) = \int_S f(x) \vert {\mathbf 1} -x\vert^s d\sigma(x)
\]
\[= \int_S \big(f(x)-f({\mathbf 1})\big)\vert {\mathbf 1}-x\vert^s d\sigma(x) + f({\mathbf 1})\int_S \vert {\mathbf 1}-x\vert^s d\sigma(x)\ .
\]
Now $\vert f(x)-f({\mathbf 1})\vert \leq C\vert x-{\mathbf 1}\vert$, so that the first integral is absolutely convergent for $s$ in a (small) neighborhood of $-(n-1)$. Hence, the contribution to the residue at $-(n-1)$ of the first integral is $0$. Now
\[\int_S \vert {\mathbf 1}-x\vert^s d\sigma(x) = 2^{n-1}\, \pi^\rho\, 2^s \frac{\Gamma(\frac{s}{2}+\rho)}{\Gamma(\frac{s}{2}+2\rho)}\ .
\]
The function on the right handside is meromorphic, has a simple pole at $s=-(n-1)=-2\rho$, with residue equal to $\displaystyle\frac{\pi^\rho}{\Gamma(\rho)}$.
\end{proof}

\begin{proposition}
\begin{equation}\label{Yam}
r_1 := Res(h_s, -(n-1)-2) = \frac{\pi^\rho}{4\,\Gamma(\rho+1)}\,  \Delta_1 \,\delta_{\mathbf 1}\ ,
\end{equation} where 
\[  \Delta_1 = \Delta-\frac{1}{4}(n-1)(n-3)
\]
is the \emph{conformal Laplacian} or \emph{Yamabe operator} on $S$.
\end{proposition}

\begin{proof} The Bernstein-Sato identity \eqref{bern} can be extended meromorphically to $\mathbb C$. Taking residue of both sides at $-(n-1)$
yields
\[ 2(n-1) Res (h_s, -(n-1)-2) = [\Delta -(\frac{n-1}{2})(\frac{n-3}{2})]\, Res(h_s,-(n-1))\ ,
\]
which, via \eqref{res0} gives \eqref{Yam}.
\end{proof}
For any $k$ in $\mathbb N$, introduce the differential operator $ \Delta_k$ on $S$ given by
\begin{equation}
 \Delta_k = \prod_{j=1}^k \big(\Delta-(\rho+j-1)(\rho-j)\big)= \prod_{j=1}^k\big(\Delta_1 +j(j-1)\big)\ .
\end{equation}

Observe that $\Delta_k$ is a polynomial of degree $k$ in $\Delta$ which is of the form $\Delta^k+$ lower order terms. Observe that $\Delta_k$ is essentially selfdajoint. 

\begin{proposition} For $k$ any positive integer,
\begin{equation*}
r_k =  \frac{\pi^\rho}{4^k \Gamma(\rho+k)\Gamma(k+1)}\,  \Delta_k\, \delta_{\bf 1}\ .
\end{equation*}

\end{proposition}

\begin{proof}
Compute the residue at $s=-(n-1)-2k$ of both sides of the Bernstein identity \eqref{bern} to get
\begin{equation*}
r_{k+1} = \frac{1}{4(\rho+k)(k+1)} \big(\Delta -(\rho+k)(\rho-k-1)\big) r_k \ .
\end{equation*}
Hence the result.
\end{proof}
Normalize the Haar measure $dk$ on $K$ such that, for any integrable function $f$ on $S$
\[\int_K f(k\,\mathbf 1) dk = \int_S f(x) \, d\sigma(x)\ .
\]
Let $f$  be a smooth function on $S$. Let $ f^\sharp$ (resp. $g^\sharp$) be the function on $K$ defined by $ f^\sharp(k) = f(k\,\mathbf 1)$.  For $f$ and $g$ two smooth functions on $S$, define the convolution $ f^\sharp\star g^\sharp$ is defined (as a function on $K$) by
\[(f^\sharp\star g^\sharp)\, (k) = \int_K  f^\sharp(l^{-1} k)g^\sharp(l) dl\ .
\]
It is a smooth function on $K$, which is right invariant by $M$, hence defines a function on $S$, denoted by $f\star g$. Suppose moreover that $f$  is invariant by $M$, hence of the form $f(x) = F(\langle x,{\mathbf 1}\rangle) $. Then, $f\star g$ is given by
\begin{equation}\label{convM}
(f\star g) (x) = \int_S F(\langle x,y\rangle)g(y) d\sigma(y)\ .
\end{equation}
The convolution of two functions in $\mathcal C^\infty(S)$ is in $\mathcal C^\infty(S)$ and the convolution can be extended to distributions on $S$. 

For $\alpha$ a complex parameter, introduce the kernel $k_\alpha$ on $S\times S$ by
\[ k_\alpha(x,y) = \vert x-y\vert^{-\rho+\alpha}\ ,\]
and the corresponding operator (formally defined by)
\[ K_\alpha f (x) = \int_S k_\alpha(x,y)f(y)d\sigma(y)\ .\]
To make proper sense, the operator $K_\alpha$ can be reinterpreted as a convolution. As, for $x,y$ in $S$,  $\vert x-y\vert =2(1-\langle x,y\rangle)$, \eqref{convM} implies
\[K_\alpha f= h_{-\rho+\alpha}\star f\ .
\]
Using the standard continuity properties of the convolution on $S$ and the results of section 2, it is easy to analytically continue the map $\alpha\longmapsto K_\alpha$  on $\mathbb C$, with simple poles at $\alpha = -\rho-2k, k\in \mathbb N$. The corresponding residues are easily computed. In fact, let $R_k = Res(K_\alpha, -\rho-2k)$. Then 
\begin{equation}
R_k f= r_k\star f =  \frac{\pi^\rho}{4^k \Gamma(\rho+k)\Gamma(k+1)} \Delta_k f\ .
\end{equation}
Now, for $\varphi, \psi$ two functions in $\mathcal C^\infty(S)$,
\[(k_\alpha, \varphi\otimes \psi) = \int_{S\times S} k_\alpha(x,y) \varphi(x)\psi(y)d\sigma(x)d\sigma(y)=(K_\alpha \varphi,\psi)\ .
\]
As above, this bilinear form on $\mathcal C^\infty(S)\times \mathcal C^\infty(S)$ can be meromorphically continued on $\mathbb C$, with same poles as above. But by Schwartz nuclear theorem, this is equivalent to say that the $\alpha\longmapsto k_\alpha$ can be meromorphically continued as a distribution (on $S\times S$)-valued function, with same poles as above, and the residues are also easy to calculate. 
\begin{proposition}\label{resk}
Let $f$ be in $\mathcal C^\infty(S\times S)$. Then the expression
\begin{equation}\label{kalpha}
\iint_{S\times S} f(x,y) \vert x-y\vert^{-\rho+\alpha} d\sigma(x)\, d\sigma(y) 
\end{equation}
originally defined for $\Re \alpha$ large enough can be continued meromorphically to $\mathbb C$, with simple poles at $\alpha = -\rho-2k, k\in \mathbb N$. The residue at $\alpha = -\rho-2k$ is given by
\[\int_S R_k^{(1)}f(x,x) d\sigma(x)\ ,
\]
where $R_k^{(1)}$ stands for the differential operator $R_k$ acting on the first variable.
 \end{proposition}
 
\noindent
{\bf Remark 1.} We will need a slightly stronger version of this result. Let $k$ be in $\mathbb N$. Then it is possible to choose $\ell$ large enough so that, if $f$ is merely in $\mathcal C^\ell (S\times S)$, then the integral \eqref{kalpha} can be meromorphically continued to $\Re(\alpha)>-\rho-2k$. This result is proved exactly the same way, but  using Proposition \ref{mero1} instead of Proposition \ref{mero0}.

\noindent
{\bf Remark 2.} The operator $R_k$ is symmetric, so that the residue can also be written as $\int_S R_k^{(2)} f (x,x) d\sigma(x)$, that is by letting $R_k$ act on the second variable.

An important property of the operators $K_\alpha$ is that it is a family of intertwining operators for the representations constructed earlier (the Knapp-Stein intertwining operators for the spherical principal series). In fact, by a change of variable (and meromorphic continuation), one shows that for any $g$ in $G$
\[ K_{-\rho+2\lambda} \circ \pi_\lambda (g) = \pi_{-\lambda}(g) \circ K_{-\rho+2\lambda}\ .
\]
Hence, taking residue on both sides at $\lambda = -k$ 
\begin{equation}\label{intw}
R_k \circ \pi_{-k}(g) = \pi_k(g)\circ R_k
\end{equation}
for any $g$ in $G$. This shows that $ \Delta_k$ is a covariant differential operator w.r.t. $(\pi_{-k}, \pi_k)$, solving the problem raised at the end of section 2.

For $k=1$, \eqref{intw} is a well-known property of the Yamabe operator on the sphere. For higher values of $k$, it corresponds to the conformal invariance of the Graham-Jenne-Mason-Sparling operators on the sphere (see \cite{gjms}, \cite{j}).

\section{Construction of singular conformally invariant trilinear forms}

We now begin the heart of this article. In \cite{co}, we studied trilinear invariant forms on $\mathcal C^\infty(S)\times \mathcal C^\infty(S)\times \mathcal C^\infty(S)$. Let $\lambda_1,\lambda_2,\lambda_3$ be three complex numbers and let $\boldsymbol \lambda = ( \lambda_1,\lambda_2,\lambda_3)$. A continuous trilinear form $\mathcal T$ on $\mathcal C^\infty(S)\times \mathcal C^\infty(S)\times \mathcal C^\infty(S)$ is said to be invariant w.r.t. $(\pi_{\lambda_1}, \pi_{\lambda_2},\pi_{\lambda_3})$ if, for any $g$ in $ G$ and $f_1,f_2,f_3$ in $\mathcal C^\infty(S)$
\[ \mathcal T(\pi_{\lambda_1}(g)f_1,\pi_{\lambda_2}(g)f_2,\pi_{\lambda_3}(g)f_3) = \mathcal T(f_1,f_2,f_3)\ .\]
For generic $\boldsymbol \lambda$ (the exact meaning of "generic" will be given in next section), we constructed a non trivial invariant trilinear form and showed that it is the unique one (up to a constant). Now we deal with the case of singular values of $\boldsymbol \lambda$.

Let $k$ be in $\mathbb N$, and set for convenience $\alpha_3 = -\rho-2k$ (notation will be explained in the next section). For $\alpha_1,\alpha_2$ in $\mathbb C^2$, define ${\mathcal T}_k = {\mathcal T}_{(\alpha_1,\alpha_2,-\rho-2k)}$ on $\mathcal C^\infty(S) \times \mathcal C^\infty(S)\times \mathcal C^\infty(S)$ by 
\begin{equation}\label{reshy3}
\begin{split}
 \mathcal T_k (f_1,f_2,f_3)= \\ \int_{S\times S} \!\!\!\!f_3(x_3) f_2(x) \Delta_k[f_1(.)\vert x_3-.\vert^{-\rho+\alpha_2}](x)&\vert x-x_3\vert^{-\rho+\alpha_1}\,d\sigma(x)\, d\sigma(x_3)\ .
\end{split}
\end{equation}

\begin{proposition}\label{invsing}
 Let $\boldsymbol \alpha = (\alpha_1,\alpha_2,\alpha_3 =-\rho-2k)$ and let $\boldsymbol \lambda = (\lambda_1,\lambda_2,\lambda_3)$ satisfying the following relations
 \begin{equation}\label{param}
\begin{split}
\alpha_1 &= -\lambda_1+\lambda_2+\lambda_3\\
\alpha_2 &= \lambda_1-\lambda_2+\lambda_3\\
\alpha_3 &= \lambda_1+\lambda_2-\lambda_3 \ .
\end{split}
\end{equation}
Then
\[\mathcal T_k(\pi_{\lambda_1}(g)f_1, \pi_{\lambda_2} (g) f_2, \pi_{\lambda_3}(g)f_3)= \mathcal T_k(f_1,f_2,f_3)
\]
for any $g$ in $G$ and $f_1,f_2,f_3 \in \mathcal C^\infty(S)$, whenever the integrals make sense.
\end{proposition}

\begin{proof}

We first need a technical lemma. It will be convenient to set,  for $f_1$ in $\mathcal C^\infty(S)$ and $x_3$ in $S$, 
\[F_{x_3}[f_1](x) = f_1(x)\vert x_3-x\vert^{-\rho+\alpha_2}\ .\]
\begin{lemma} For any $g$ in $G$, $f_1$ in $\mathcal C^\infty(S)$ and $x_3$ in $S$,
\begin{equation}\label{astuce}
 F_{x_3}[\pi_{\lambda_1}(g)f_1] =  \kappa(g,y_3)^{-\frac{\rho}{2}+\frac{\alpha_2}{2}}\,\pi_{-k}(g)F_{y_3}[f_1] \ ,
\end{equation}
where  $x_3=g(y_3)$.
\end{lemma}
\begin{proof}
\[LHS = f_1(g^{-1}(x))\vert x_3-x\vert^{-\rho+\alpha_2}\kappa(g^{-1},x)^{\rho+\lambda_1}
\]
\[ =f_1(g^{-1}(x)) \kappa(g,y_3)^{-\frac{\rho}{2}+\frac{\alpha_2}{2}}\vert y_3-g^{-1}(x)\vert^{-\rho+\alpha_2} \kappa(g^{-1},x)^{\frac{\rho}{2} -\frac{\alpha_2}{2}+\rho+\lambda_1}
\]
using \eqref{cov}.
Now 
\[\frac{\rho}{2} -\frac{\alpha_2}{2}+\rho+\lambda_1= \rho-k\ .
\]
from which \eqref{astuce} follows.
\end{proof}
With the notation introduced previously, 
\[ T_k(f_1,f_2,f_3) =\iint f_3(x_3)\,f_2(x) \,\Delta_k F_{x_3}[f_1](x)\, \vert x-x_3\vert^{-\rho+\alpha_1}\,d\sigma(x)\,d\sigma(x_3),
\]
so that

\[T_k(\pi_{\lambda_1}(g)f_1, \pi_{\lambda_2}(g)f_2,\pi_{\lambda_3}(g)f_3)=\]\[\iint f_3(g^{-1}(x_3))f_2(g^{-1}(x))\Delta_k \{F_{x_3}[\pi_{\lambda_1}(g)f_1]\}(x)\vert x-x_3\vert^{-\rho+\alpha_1}\dots\]\[\dots\kappa(g^{-1}, x)^{\rho+\lambda_2}\kappa(g^{-1},x_3)^{\rho+\lambda_3} d\sigma(x)d\sigma(x_3)
\]
Now use \eqref{astuce} and \eqref{intw}, make the change of variables $x=g(y)$ and $x_3=g(y_3)$ to get
\[T_k(\pi_{\lambda_1}(g)f_1, \pi_{\lambda_2}(g)f_2,\pi_{\lambda_3}(g)f_3)\]
\[ =\iint f_3(y_3)f_2(y) \pi_k(g)\Delta_k\{F_{y_3}[f_1]\}(g(y))\vert y-y_3\vert^{-\rho+\alpha_1}\dots\]
\[\dots\kappa(g,y_3)^{-\frac{\rho}{2}+\frac{\alpha_2}{2}-\frac{\rho}{2}+\frac{\alpha_1}{2}-\rho-\lambda_3+2\rho}\kappa(g, y)^{-\frac{\rho}{2}+\frac{\alpha_1}{2}-\rho-\lambda_2+2\rho}\, d\sigma(y)d\sigma(y_3)
\]
\[ =\iint f_3(y_3)f_2(y)\Delta_kF_{y_3}[f_1](y) \vert y-y_3\vert^{-\rho+\alpha_1}\dots\]
\[\kappa(g,y_3)^{-\frac{\rho}{2}+\frac{\alpha_2}{2}-\frac{\rho}{2}+\frac{\alpha_1}{2}-\rho-\lambda_3+2\rho}\kappa(g, y)^{-\frac{\rho}{2}+\frac{\alpha_1}{2}-\rho-\lambda_2+2\rho-\rho-k}\, d\sigma(y)d\sigma(y_3)\ .
\]
Now
\[-\frac{\rho}{2}+\frac{\alpha_2}{2}-\frac{\rho}{2}+\frac{\alpha_1}{2}-\rho-\lambda_3+2\rho= 0
\]
\[-\frac{\rho}{2}+\frac{\alpha_1}{2}-\rho-\lambda_2+2\rho-\rho-k =0,
\]
where in the second line we use the condition $\alpha_3 = -\rho-2k$. So the last integral reduces to 
\[\iint f_3(y_3)f_2(y)\Delta_kF_{y_3}[f_1](y) \vert y-y_3\vert^{-\rho+\alpha_1} d\sigma(y) d\sigma(y_3)\ ,
\]
which shows Proposition \ref{invsing}.

\end{proof}
Having proved (formally) the invariance of the form $\mathcal T_k$, we now study  the convergence of the integral and its meromorphic continuation. 

 \begin{theorem}\label{singmero}
  Let $k$ is in $\mathbb N$. Then the trilinear form $\mathcal T_k=\mathcal T_{ \alpha_1, \alpha_2, -\rho-2k}$ originally defined as a convergent integral for $\Re \alpha_1$ and $\Re \alpha_2$ large enough can be extended meromorphically to $\mathbb C^2$, with poles along the lines
 \[\alpha_1+\alpha_2 = 2k-2l, l\in \mathbb N\ .
 \]
The trilinear form $\mathcal T_k$ is invariant w.r.t. $(\pi_{\lambda_1}, \pi_{\lambda_2},\pi_{\lambda_3})$, where $\boldsymbol \lambda = (\lambda_1,\lambda_2,\lambda_3)$ is associated to $\boldsymbol \alpha$ by the relations \eqref{param}.
 \end{theorem}
\begin{proof}
We first need a technical lemma.
\begin{lemma}\label{derker}
Let $\varphi$ be in $\mathcal C^\infty(S\times S)$. There exists a (unique) function $\psi(x,y,s)$ which is $\mathcal C^\infty$ in $x$ and $y$ and holomorphic in $s$ such that, for all $x\neq y$ in $S\times S$
\[\Delta_x\,[\vert x-y\vert^s \varphi(x,y)] = \vert x-y\vert^{s-2} \psi(x,y,s)\]
Moreover, the function $\psi$ depends continuously on $\varphi$.
\end{lemma}
\begin{proof} Recall the formula
\[\Delta(fg) = \Delta f\ g+ 2 \,\overrightarrow{grad}\,f.\,\overrightarrow{grad}\,g+f \Delta g \ .
\]
By \eqref{bern} and the fact that $\Delta$ commutes with the action of $
K$ 
\[\Delta_x(\arrowvert x-y\arrowvert ^s) = \big(-\frac{s}{2}(\frac{s}{2}+n-2)\vert x-y\vert^2 +s(s+n-3)\big)\vert x-y\vert^{s-2}\ .
\]
Moreover,
\[\overrightarrow{grad}_x\, (\vert x -y\vert^s) = \frac{s}{2} \vert x-y\vert^{s-2}\,\overrightarrow{grad}_x\,( \vert x-y\vert^2)\ .
\]
The statement of the lemma is a consequence of these three fromul\ae.
\end{proof}

Now, by repeated uses of Lemma \ref{derker}, it is possible to write
\[\Delta_k[f_1(.)\vert x_3-.\vert^{-\rho+\alpha_2}](x)= \sum_{\ell = 0}^k\vert x_3-x\vert^{-\rho+\alpha_2-2\ell}\psi_\ell(x_3,x,\alpha_2)
\]
where the $\psi_\ell$ are smooth functions on $S\times S$, holomorphic in $\alpha_2$, depending continuously on $f_1$. Hence $\mathcal T_k(f_1,f_2,f_3)$ can be written as a sum of integrals of the form

\[ \int_{S\times S} f_2(x)f_3(x_3) \psi_j(x_3,x,s)\vert x-x_3\vert^{-2\rho+\alpha_2+\alpha_1-2j}\, d\sigma(x_3)d\sigma(x)\]
where $ j =0,1,\dots,k$, and $\psi_j$ is a smooth function on $S\times S$, holomorphic in $s$ and depending continuously on $f_1$. Such integrals can be meromorphically continued by use of Proposition \ref{resk} and the localization of the poles also follows.
\end{proof} 
 
\section {Residues of $K_{\boldsymbol \alpha}$ along a hyperplane of the first kind}
Now recall the construction and results of \cite{co}. 
 Let ${\boldsymbol \alpha} = (\alpha_1,\alpha_2,\alpha_3)$ be in $\mathbb C^3$. Consider the kernel $K_{\boldsymbol \alpha}$ on $S\times S\times S$ defined by
 \begin{equation*}
K_{\boldsymbol \alpha}(x_1,x_2,x_3) = k_{\alpha_1}(x_2,x_3)\,  k_{\alpha_2}(x_3,x_1)\,k_{\alpha_3}(x_1,x_2)
\end{equation*}
and the associated trilinear form $\mathcal K_{\boldsymbol \alpha}$ defined by
 \begin{equation*}
\mathcal K_{\boldsymbol \alpha}(f_1,f_2,f_3) =
\int_ {S\times S\times S}\!\!\!\!\!\!\!\!\!\!\!\! K_{\boldsymbol \alpha}(x_1,x_2,x_3)f_1(x_1)f_2(x_2)f_3(x_3) d\sigma(x_1)\, d\sigma(x_2)\, d\sigma(x_3)
\end{equation*}

 \begin{proposition} Let $\boldsymbol \alpha = (\alpha_1,\alpha_2,\alpha_3)$ be in $\mathbb C^3$.
Then the integral $\mathcal K_{\boldsymbol \alpha}(f_1,f_2,f_3)$ is absolutely convergent if $\Re \alpha_j > -\rho, j=1,2,3$ and  $\Re(\alpha_1+\alpha_2+\alpha_3) > -\rho$. The mapping $\boldsymbol \alpha \longmapsto \mathcal K_{\boldsymbol \alpha}(f_1,f_2,f_3)$ can be extended meromorphically to $\mathbb C^3$ , with simple poles along the planes (in $\mathbb C^3$)
\begin{equation}
\begin{split}
\alpha_1 &=-\rho-2k_1, \quad k_1\in \mathbb N\\
\alpha_2 &= -\rho-2k_2, \quad k_2 \in \mathbb N\\
\alpha_3 &= -\rho-2k_3, \quad k_3 \in \mathbb N\\
\alpha_1+\alpha_2+\alpha_3 &= -\rho-2k, \quad k\in \mathbb N
\end{split}
\end{equation}
For $\boldsymbol \alpha$ outside of these planes, the trilinear form $\mathcal K_{\boldsymbol \alpha} (f_1,f_2,f_3)$ thus defined is  invariant  w.r.t.  the representations $(\pi_{\lambda_1}, \pi_{\lambda_2}, \pi_{\lambda_3})$, where $\boldsymbol \lambda = (\lambda_1,\lambda_2,\lambda_3)$ and $\boldsymbol \alpha = (\alpha_1,\alpha_2,\alpha_3)$ are related by the relations \eqref{param}.
\end{proposition}

From the results in \cite{co}, more can be said about the residues of $\mathcal K_{\boldsymbol \alpha}$. The poles come into four families of parallel and equidistant planes in $\mathbb C^3$. The first three (said to be \emph{of type 1}) have some sort of symmetry under a permutation of $\alpha_1,\alpha_2,\alpha_3$, so that it suffices to study the case where $\alpha_3 = -\rho-2k, k\in \mathbb N$. Then the residue is a distribution supported in $\{(x,x,x_3), x, x_3\in S\}$. For the fourth family (\emph{type 2}), the residue is a distribution supported in the "diagonal" of $S\times S\times S$, i.e. the set $\{x,x,x\}, x\in S$. This last family is not considered in the present paper.

 We now proceed to the determination of the residue of the form $\mathcal K_{\boldsymbol \alpha}$ at a pole of the form $\boldsymbol \alpha = (\alpha_1,\alpha_2,-\rho-2k)$ for $\alpha_1,\alpha_2$ generic.
\begin{theorem}\label{regres}
Let $\mathbf \alpha^0 = (\alpha_1,\alpha_2,\alpha_3^0)$ be a generic point in the hyperplane $\alpha_3^0 = -\rho-2k$. Then
\begin{equation}\label{reshy3}
\begin{split}
Res\big(K_{\mathbf \alpha}(f_1,f_2,f_3)&,{ \mathbf\alpha}^0\big) = \\ \int_{S\times S} \!\!\!\!f_3(x_3) f_2(x) R_k[f_1(.)\vert x_3-.\vert^{-\rho+\alpha_2}](x)&\vert x-x_3\vert^{-\rho+\alpha_1}\,d\sigma(x)\, d\sigma(x_3)\ .
\end{split}
\end{equation}
More precisely, the right handside integral which converges for $\Re \alpha_1$ and $\Re \alpha_2$ large engouh can be extended meromorphically to $\mathbb C^2$ with poles contained in the lines 
\[ \alpha_1+\alpha_2 = 2k-2l,\  l\in \mathbb N\ .
\]
The left handside, a priori defined for $\alpha_1, \alpha_2$ outside of the lines 
\[\alpha_1 = -\rho-2k_1,\quad \alpha_2= -\rho-2k_2,\quad \alpha_1+\alpha_2 = 2k-2l_3\ ,
\]
$k_1,k_2,l_3\in \mathbb N$, 
coincides with the left handside.
\end{theorem}

\begin{proof}
As $R_k$ and $\Delta_k$ differ by a non vanishing constant, the right handside of \eqref{reshy3} depends  meromorphically on $(\alpha_1,\alpha_2)$ (see Theorem  \ref{singmero}) so that, by properties of analytic continuation, it is enough to verify the equality when $\Re \alpha_1$ and $\Re \alpha_2$ are large enough. In this spirit, we have the following technical lemma.
\begin{lemma} Let $f_3$ be in $\mathcal C^\infty(S)$. For $x_1,x_2$ in $S$, let
\begin{equation}\label{iintF3}
F_3(x_1,x_2) = \int_S f_3(x_3) \vert x_2-x_3\vert^{-\rho+\alpha_1} \vert x_3-x_1\vert^{-\rho+\alpha_2}\, d\sigma(x_3)\ .
\end{equation}
 Let $l$ be in $\mathbb N$. Then for $\Re \alpha_1$ and $\Re \alpha_2$ large enough, the function $F_3$ is in $\mathcal C^l(S\times S)$, and the map $f_3\mapsto F_3$ is continuous from $\mathcal C^\infty(S)$ to $\mathcal C^l(S\times S)
 $.
\end{lemma}
\begin{proof} Choose, at it is possible, $\Re \alpha_1$ and $\Re \alpha_2$ large enough, so that the kernel 
\[ \vert x_2-x_3\vert^{-\rho+\alpha_1} \vert x_3-x_1\vert^{-\rho+\alpha_2}
\]
is, as a function of three variables, everywhere $l$-times continuously differentiable. Then the statement follows by the rule of differentiation under integral sign and standard estimates.
\end{proof}

Now, let $f_1,f_2,f_3$ be three functions in $\mathcal C^\infty(S)$. Let $\Re\alpha_1$ and $\Re\alpha_2$ be large enough, so that the function $F_3$ defined by \eqref{iintF3} is in $\mathcal C^l(S\times S)$ for some (large) $l$. For $\alpha_3$ near $\alpha_3^0$, but $\alpha_3\neq \alpha_3^0$, the kernel $\vert x_1-x_2\vert^{-\rho+\alpha_3}$ corresponds to a distribution on $S\times S$, depending meromorphicaly on $\alpha_3$ and hence (see Lemma \ref{resk} and the remark inside the proof) the expression
\[\int_{S\times S} f_1(x_1)f_2(x_2) F_3(x_1,x_2) \vert x_1-x_2\vert^{-\rho+\alpha_3} \, d\sigma(x_1)\,d\sigma(x_2)
\]
depends meromorphically on $\alpha_3$. At $\alpha_3 = -\rho-2k$, the expression has a simple pole and its residue  is given by 
\[ \int_S R_k[ f_1(.)f_2(x) F_3(.,x)](x)\,d\sigma(x)
\]
Now, we differentiate under the integral sign (assuming, as it is possible, that $F_3$ is $2k$-times continuously differentiable), to get 
\[ Res(K_{\mathbf \alpha}(f_1,f_2,f_3), \mathbf \alpha^0)\] =
\[=\int_{S\times S} f_2(x)f_3(x_3)R_k[f_1(.) \vert x_3-.\vert^{-\rho+\alpha_2}](x) \vert x-x_3\vert^{-\rho+\alpha_1}d\sigma(x_3)\, d\sigma(x)\ .
\]
The equality \eqref{reshy3} is thus proved for $\Re \alpha_1$ and $\Re \alpha_2$ large enough.
\end{proof}

The left handside of \eqref{reshy3} is  {\it a priori} defined in the intersection of the domain of definition of $K_{\mathbf \alpha}$ and the plane $\alpha_3 = -\rho-2k$, that is to say on $\mathbb C^2$ only outside of the lines
\[ \alpha_1 = -\rho-2 l_1,\quad \alpha_2 = -\rho-2l_2,\quad \alpha_1+\alpha_2 = 2 k -2l_3, \quad \]
where $l_1,l_2,l_3$ are in $\mathbb N$. A consequence of Theorem \ref{regres} is that it can be extended to a larger domain. This can be illustrated on the evaluation of $K_{\mathbf \alpha}$ for $f_1,f_2,f_3$ equal to the constant function $1$. Then (see \cite{ckop}), up to a constant, 
$K_{\mathbf \alpha}(1,1,1)$ is equal to 
\[\frac{\Gamma\big(\frac{\alpha_1+\alpha_2+\alpha_3+\rho}{2}\big)\,\Gamma\big(\frac{\alpha_1+\rho}{2}\big)\,\Gamma\big(\frac{\alpha_2+\rho}{2}\big)\,\Gamma\big(\frac{\alpha_3+\rho}{2}\big)}{\Gamma\big(\rho+\frac{\alpha_2+\alpha_3}{2}\big)\,\Gamma\big(\rho+\frac{\alpha_3+\alpha_1}{2}\big)\,\Gamma\big(\rho+\frac{\alpha_1+\alpha_2}{2}\big)}\ .
\]
The residue at $\alpha_3 = -\rho-2k$ is equal to
\[\frac{(-1)^k}{k!}\ \frac{\Gamma\big(\frac{\alpha_1+\alpha_2}{2}-k\big)\Gamma\big(\frac{\alpha_1+\rho}{2}\big)\Gamma\big(\frac{\alpha_2+\rho}{2}\big)}{\Gamma\big(\frac{\rho+\alpha_2}{2}-k\big)\Gamma\big(\frac{\rho+\alpha_1}{2}-k\big)\Gamma\big(\rho+\frac{\alpha_1+\alpha_2}{2}\big)}\ .
\]
which equals 
\[\frac{(-1)^k}{k!}\big(\frac{\rho+\alpha_1}{2}-1\big)\dots\big(\frac{\rho+\alpha_1}{2}-k\big)\big(\frac{\rho+\alpha_2}{2}-1\big)\dots\big(\frac{\rho+\alpha_2}{2}-k\big)
\frac{\Gamma\big(\frac{\alpha_1+\alpha_2}{2}-k\big)}{\Gamma\big(\rho+\frac{\alpha_1+\alpha_2}{2}\big)}
\]
and this expression has simple poles exactly along the lines \[\alpha_1+\alpha_2 = 2k-2l, l\in \mathbb Z\ .\]

\bigskip
\footnotesize{ \noindent Jean-Louis Clerc \\Institut \'Elie Cartan
(CNRS UMR 7502), Universit\'e Henri Poincar\'e Nancy 1\\
B.P. 70239, F-54506 Vandoeuvre-l\`es-Nancy, France.\\

\begin{thebibliography}{99}\itemsep=-.2pc


\bibitem{co} Clerc J-L. and \O rsted B., {\it Conformally invariant trilinear forms on the sphere}, to appear in Ann. Instit. Fourier

\bibitem{ckop} Clerc J-L., Kobayashi T., \O rsted B. and Pevzner M., {\it Generalized Bernstein-Reznikov integrals}, to appear in Math. Annalen

\bibitem{d} Dobrev V.K., {\it Canonical construction of interetwining differential operators associated with representations of real semisimple Lie groups}, Rep. Math. Phys. {\bf 25} (1988), 159--181

\bibitem{gjms} Graham C., Jenne R., Mason L. and Sparling G., {\it Conformally invariant powers of the Laplacian. I. Existence},  J. London Math. Soc.   {\bf 46} (1992), 557Ð565. 

\bibitem{j} Juhl A. {\it Families of conformally covariant differential operators, Q-curvature and holography}, Progress in Mathematics, {\bf 275}, Birkh\" auser Verlag (2009)

\bibitem{k} Kostant B., {\it Verma modules and the existence of quasi-invariant differential operators}, Lectures Notes in Mathematics {\bf 486} (1975) 101--128, Sringer Verlag, Berlin

\bibitem{lep1} Lepowsky J., {\it Conical vectors in induced modules}, Trans. Amer. Math. Soc. {\bf 208}, (1975), 219--272

\bibitem{lep2}  Lepowsky J., {\it On the uniqueness of conical vectors}, Proc. Amer. Math. Soc. {\bf 57} (1976), 217--220

\bibitem{tak} Takahashi R., {\it Sur les fonctions unitaires des groupes de Lorentz g\'en\'eralis\'es}, Bull. Soc. Math. France {\bf 91} (1963) 289--433

\end{thebibliography}
\end{document}